\documentclass[11pt,reqno]{amsart}
\usepackage{amsmath,amsfonts, amssymb, amsthm}
\usepackage{mathrsfs}
\usepackage{graphicx, color,hyperref}

\def\qed{\hfill \rule{4pt}{7pt}}

\def\l{\left}
\def\r{\right}
\def\bg{\bigg}
\def\({\bg(}
\def\){\bg)}
\def\t{\text}
\def\f{\frac}

\def\ls{\leqslant}

\def\bi{\binom}

\def\qed{\hfill \rule{4pt}{7pt}}
\newtheorem{thm}{Theorem}[section]

\theoremstyle{remark} 
\newtheorem{remark}{Remark}[section]

\numberwithin{equation}{section}

\begin{document}
\hbox{Preprint}

\title[Evaluations of some series via the WZ method]
      {Evaluations of some series via the WZ method}


\author[Q.-H. Hou]{Qing-Hu Hou}
\address{School of Mathematics, Tianjin
University, Tianjin 300072, People's Republic of China}
\email{{\tt qh\_hou@tju.edu.cn}
\newline\indent
{\it Homepage}: {\tt http://faculty.tju.edu.cn/HouQinghu/en/index.htm}}

\author[Z.-W. Sun]{Zhi-Wei Sun}
\address{School of Mathematics, Nanjing
University, Nanjing 210093, People's Republic of China}
\email{{\tt zwsun@nju.edu.cn}
\newline\indent
{\it Homepage}: {\tt http://maths.nju.edu.cn/\lower0.5ex\hbox{\~{}}zwsun}}

\keywords{Binomial coefficients, combinatorial identities, Riemann's zeta function, infinite series.
\newline \indent Supported by the Natural Science Foundation of China (grant no. 11921001 and 12371004, respectively).}

\begin{abstract} In this paper, we evaluate some series via the WZ method, and confirm several previous conjectures. For example, we prove the following two identities conjectured by the second author:
$$\sum_{k=0}^{\infty} \frac{(28k^2 + 10k + 1) \binom{2k}{k}^5}{(6k + 1)(-64)^k \binom{3k}{k} \binom{6k}{3k}} = \frac{3}{\pi}$$
and
$$\sum_{k=1}^\infty \f{d^4}{dk^4}\left(\frac{(21k-8)\Gamma(k+1)^2}{k^3\Gamma(2k+1)}\right)=\frac{1959}2\zeta(6)-432\zeta(3)^2.
$$
\end{abstract}
\maketitle

\section{Introduction}

A series $\sum_na_n$ is called  {\it hypergeometric} if the ratio $a_{n+1}/a_n$ of two consecutive terms is a rational function of $n$. More generally, a multi-variate function  $F(n_1, \ldots,n_r)$ is {\it hypergeometric}  if all the ratios
\[
\frac{F(n_1+1,n_2,\ldots,n_r)}{F(n_1,n_2,\ldots,n_r)}, \ \ldots, \frac{F(n_1,\ldots,n_{r-1}, n_r+1)}{F(n_1,n_2,\ldots,n_r)}
\]
are rational functions of $n_1,n_2,\ldots,n_r$. Let $\Delta_n f(n) = f(n+1) - f(n)$ denote the difference operator. A pair of hypergeometric terms $F(n,k)$ and $G(n,k)$ is called a {\it Wilf-Zeilberger {\rm (}WZ\/{\rm )} pair}  \cite{WZ92, Pet96} if
\begin{equation}\label{WZ}
	\Delta_n F(n,k) = \Delta_k G(n,k).
\end{equation}
In this case, we call $G(n,k)$ the {\it WZ mate} of $F(n,k)$. Summing over $n,k$ for both sides of \eqref{WZ}, we obtain
\begin{equation}\label{keyid}
 \sum_{n=0}^\infty G(n,0)  = \sum_{k=0}^\infty F(0,k) + \sum_{n=0}^\infty g(n) - \lim_{n \to \infty} \sum_{k=0}^\infty F(n,k),
\end{equation}
where $g(n) = \lim_{k \to \infty} G(n,k)$.
When the series on the right hand side of \eqref{WZ} can be evaluated, we will get the evaluation of the series on the left hand side. This provided us a powerful method, called the {\it WZ method}, for evaluating hypergeometric series.

Based on the WZ method, Amdeberhan and Zeilberger \cite{AZ97} presented a  fast convergent series for $\zeta(3)$:
\[
\zeta(3) = \sum_{n=0}^\infty (-1)^n \frac{n!^{10} (205n^2+250n+77)}{64 (2n+1)!^5}.
\]
Guillera \cite{Gui08} proved $10$ extended Ramanujan-type series, one of which is
\[
\sum_{n=0}^\infty  \frac{(6(n+a)+1)(a+\frac{1}{2})_n^3}{4^n(a+1)_n^3}  = 8a \sum_{n=0}^\infty \frac{(\frac{1}{2})_n^2}{(a+1)_n^2}.
\]
In particular, this formula with $a = \frac{1}{2}$ yields
\[\sum_{k=1}^\infty\f{(3k-1)16^k}{k^3\bi{2k}k^3}=
\sum_{n=0}^\infty \frac{1}{2^{2n}} \frac{(1)_n^3}{(\frac{3}{2})_n^3} (3n+2) = \frac{\pi^2}{4}.
\]
For $q$-analogues of this formula and similar ones, one may consult 
Hou, Krattenthaler and Sun \cite{HKS}, and Campbell \cite{Camp26}.

Gessel \cite{Gess94}  provided a systematic method to generate WZ pairs. Namely, given a hypergeometric summation formula
\[
\sum_k F(k, a,b,\ldots) = H(a,b,\ldots),
\]
the term
\[
F_0(k,a,b,\ldots) = \frac{F(k,a,b,\ldots)} {H(a,b,\ldots)}
\]
often possesses the following property: For all integers $K,A,B,\ldots$,
\[
F(n,k) =F_0 (Kn+k_0+k, An+a, Bn+b,\ldots)
\]
has a WZ mate. Au \cite{Au25a, Au25b} named such a hypergeometric term $F_0(k,a,b,\ldots)$ a {\it WZ seed}. By selecting suitable WZ seeds and parameters, he presented many hypergeometric identities and confirmed several conjectures posed by Sun \cite{Sun24, Sun26}.

In this paper, we use the WZ method to evaluate some hypergeometric series and their derivatives.
We construct WZ pairs by selecting suitable parameters from WZ seeds. Moreover, we utilize the following transformation of WZ pairs
\[
(F(n,k), G(n,k)) \to (F(n,-k-1), -G(n,-k))
\]
and Euler's reflection formula
\[
\Gamma(x) \Gamma(1-x) = \frac{\pi}{\sin (\pi x)}.
\]
By applying these transformations, we try to construct a WZ pair $(F(n,k), G(n,k))$ such that the right hand side of \eqref{keyid} is well-defined and easy to evaluate.

Via the WZ method, we obtain the following two theorems originally conjectured by the second author (cf. \cite[Conjectures 5.27(i) and 5.35(ii)]{Sun26}).

\begin{thm}\label{th1}
We have
\begin{equation}\label{id-1}
\sum_{k=1}^{\infty} \frac{(-1)^{k-1}(28k^2 -8k+1) \binom{2k}{k}^2}{(2k-1)^2 k \binom{6k}{3k} \binom{3k}{k}} = 2 \log 2.
\end{equation}
\end{thm}

\begin{thm}\label{th2}
	We have
	\begin{equation}\label{2-1}
		\sum_{k=0}^{\infty} \frac{(28k^2 + 10k + 1) \binom{2k}{k}^5}{(6k + 1)(-64)^k \binom{3k}{k} \binom{6k}{3k}} = \frac{3}{\pi},
	\end{equation}
	\begin{equation}\label{2-2}
		\sum_{k=0}^{\infty} \frac{\big((28k^2 + 10k + 1)(2H_{2k} - 3H_k) + 20k + 4\big) \binom{2k}{k}^5}{(6k + 1)(-64)^k \binom{3k}{k} \binom{6k}{3k}} = \frac{18 \log 2}{\pi},
	\end{equation}
	\begin{equation}\label{2-3}
		\sum_{k=0}^{\infty} \frac{\big((28k^2 + 10k + 1)(2H_{6k} - H_{3k} - 3H_k) + f(k)\big) \binom{2k}{k}^5}{(6k + 1)(-64)^k \binom{3k}{k} \binom{6k}{3k}} = \frac{30 \log 2}{\pi},
	\end{equation}
	where $H_n=\sum_{0<k\ls n}\f1k$ and $f(k) = 4(138k^2 + 52k + 5)/(3(6k + 1))$.
\end{thm}

Our second observation is that if $(F(n,k), G(n,k))$ is a WZ-pair then so is
\[
\left( \frac{\partial^m}{\partial n^m} F(n,k),  \frac{\partial^m}{\partial n^m} G(n,k)  \right)
\]
for each $m=0,1,\ldots$. Under suitable convergent condition, we will get
\[
\sum_{n=0}^\infty \frac{\partial^m}{\partial n^m} G(n,0) = \sum_{k=0}^\infty \frac{\partial^m}{\partial n^m} F(0,k).
\]
When $F(0,k)$ is a rational function in $k$, $\frac{\partial^m}{\partial n^m} F(0,k)$ will be a polynomial in harmonic numbers with rational functions coefficients in $k$. These sums can be evaluated using Au's package on multiple zeta functions (cf. \cite{Au25b}).
Along this approach, we are able to evaluate several hypergeometric series and their derivatives,
and confirm many recent conjectures announced in recent postings on {\tt MathOverflow} (such as \cite{Chen,More}) and the preprint of Sun \cite{SunNew}.

\begin{thm}\label{f1}
	For $k>0$, let
	$$f_1(k)=\frac{ \Gamma(k+1)^2}{k^2 \Gamma(2k+1)}.$$
	Then
	\begin{align} \sum_{k=1}^\infty f_1^{(2)}(k)&=\frac{31}{6} \zeta(4), \label{2-2}
		\\ \sum_{k=1}^\infty f_1^{(3)}(k)&=-38 \zeta(5)+8 \zeta(2) \zeta(3), \label{2-3}
		\\ \sum_{k=1}^\infty f_1^{(4)}(k) &= -32 \zeta(3)^2+ \frac{979}{6} \zeta(6), \label{2-4}
		\\ \sum_{k=1}^\infty f_1^{(5)}(k) &=  760 \zeta(2) \zeta(5)+360 \zeta(3) \zeta(4)- 2465 \zeta(7), \label{2-5}
	\end{align}
	\begin{equation}
 \sum_{k=1}^\infty f_1^{(6)}(k) = \frac{110483}{12} \zeta(8) - 5088 \zeta(3) \zeta(5) + 960 \zeta(2) \zeta(3)^2 - 672\zeta(3,5), \label{2-6}
	\end{equation}
	\begin{multline}
	 \sum_{k=1}^\infty f_1^{(7)}(k) = - 9744703 \zeta(9) + 103530 \zeta(7) \zeta(2) + 33180 \zeta(6) \zeta(3) \\
	 + 71820 \zeta(5) \zeta(4) - 4480 \zeta(3)^3, \label{2-7}
	\end{multline}		
	\begin{multline}
	\sum_{k=1}^\infty f_1^{(8)}(k) = \frac{13510461}{10} \zeta(10) - 967680 \zeta(3) \zeta(7) + 322560 \zeta(2) \zeta(3) \zeta(5) \\
	+ 120960 \zeta(4) \zeta(3)^2 - 618480 \zeta(5)^2 - 37632 \zeta(2) \zeta(5,3) + 49200 \zeta(7,3), \label{2-8}
\end{multline}	
where
$$
\zeta(3,5) = \sum_{k_1>k_2>0}\f1{k_1^3 k_2^5}, \ \ 
\zeta(5,3)=\sum_{k_1>k_2>0}\f1{k_1^5k_2^3}\ \ \t{and}\ \ 
	\zeta(7,3)=\sum_{k_1>k_2>0}\f1{k_1^7k_2^3}.$$
\end{thm}
\begin{remark} It is well known that 
$$\sum_{k=1}^\infty f_1(k)=\f{\pi^2}{18}=-\f{\zeta(2)}3.$$
Note also that 
$$\sum_{k=1}^\infty f_1(k)=-2\sum_{k=1}^\infty\f{H_{2k}-H_k+1/k}{k^2\bi{2k}k}=-\f43\zeta(3)$$
by the conjectural formulas (3.2) and (3.3) of Sun \cite{S15} confirmed by Ablinger \cite{Abl}.
 In March 2026,  Chen \cite{Chen} conjectured the formulas \eqref{2-2}-\eqref{2-5}, and then
 Ardonne, Cohen and Chen H guessed \eqref{2-7}, \eqref{2-6} and \eqref{2-8}, respectively.
\end{remark}

\begin{thm}\label{f3}
	For $k>0$, let
	$$f_2(k)=\frac{e^{\pi i k}\Gamma(k+1)^2}{k^3 \Gamma(2k+1)}.$$
	Then
	\begin{align}\sum_{k=1}^\infty f_2''(k)&=\frac{\pi^5}{25}i+\frac{8}{15}\pi^2\zeta(3)-\frac{52}5\zeta(5), \label{3-1}
		\\5\sum_{k=1}^\infty f_2^{(3)}(k)&=4\pi i(\pi^2\zeta(3)-39\zeta(5))-6(2\zeta(3)^2+3\zeta(6)), \label{3-2}
		\\\f5{12}\sum_{k=1}^\infty f_2^{(4)}(k)&=4(52\zeta(2)\zeta(5)-9\zeta(3)\zeta(4)-58\zeta(7))
		+\pi i(57\zeta(6)-4\zeta(3)^2), \label{3-3}
	\end{align}
	\begin{equation}\label{3-4}
		\begin{aligned}\f16\sum_{k=1}^\infty f_2^{(5)}(k)=&\ 32\zeta(5,3)-48\zeta(3)\zeta(5)+32\zeta(2)\zeta(3)^2-303\zeta(8)
			\\&\ +8\pi i(26\zeta(2)\zeta(5)-\zeta(3)\zeta(4)-58\zeta(7)),
		\end{aligned}
	\end{equation}
	and
	\begin{equation}\label{3-5}
		\begin{aligned}\sum_{k=1}^\infty f_2^{(6)}(k)=&\ 96\l(696\zeta(2)\zeta(7)+3\zeta(3)\zeta(6)-2\zeta(3)^3-351\zeta(4)\zeta(5)-598\zeta(9)\r)
			\\&\ +36\pi i\l(161\zeta(8)+16\zeta(2)\zeta(3)^2-48\zeta(3)\zeta(5)+32\zeta(5,3)\r).
		\end{aligned}
	\end{equation}
\end{thm}

\begin{remark}
	 It is well known that
	$$\sum_{k=1}^\infty f_2(k)=\sum_{k=1}^\infty\frac{(-1)^k}{k^3\binom{2k}k}=-\frac25\zeta(3)$$
	as discovered by R. Ap\'ery. In March 2026 the second author \cite{More}
announced \eqref{3-1} and \eqref{3-2} without proof. He also guessed that
	$$\sum_{k=1}^\infty f_2'(k)=\frac{9}{5}\zeta(4)-\frac{2}5\pi\zeta(3)i,$$
	but C. Givord (cf. \cite{More}) reduced this to the identity
	$$\sum_{k=1}^\infty \f{(-1)^k}{k^3\bi{2k}k}\l(H_{2k}-H_k+\f3{2k}\r)=-\f 9{10}\zeta(4)$$
	which follows from the conjectural formulas (3.13) and (3.14)] of Sun \cite{S15} confirmed by Ablinger \cite{Abl}.
\end{remark}

\begin{thm}\label{f19} For $x>-1/6$, define
	$$f_{3}(x) = \frac { \left( 21\,{x}^{2}+27\,x+8 \right)   \Gamma  \left( 3\,
		x+1 \right)  ^{2}}{ \left( 2\,x+1 \right)  \left( 6\,x+1
		\right)  \left( 6\,x+5 \right) \Gamma  \left( 6\,x+1 \right) }
	. $$
	Then
	\begin{align}\sum_{k=0}^\infty f_{3}'(k)&=-12 L_{-3}(2),
		\\\sum_{k=0}^\infty f_{3}''(k)&= \frac{32 \pi^3}{9 \sqrt{3}},
		\\\sum_{k=0}^\infty f_{3}'''(k)&=54( 2 \pi^2 L_{-3}(2) -27 L_{-3}(4)),
	\end{align}
where
$$L_{-3}(s)=\sum_{k=1}^\infty\f{(\f{-3}k)}{k^s}=\sum_{n=0}^\infty \l(\f1{(3n+1)^s}-\f1{(3n+2)^s}\r)
\ \t{for}\ s=0,1,2,\ldots.$$
\end{thm}
\begin{remark}
   Note that
	$$\sum_{k=0}^\infty f_{3}(k)=\sum_{k=0}^\infty \frac{21k^2+27k+8}{(2k+1)(6k+1)(6k+5) {6k \choose 3k}} = \frac{8 \pi}{9 \sqrt{3}}$$
	by \cite[Example 115]{CZ}.
\end{remark}

\begin{thm}\label{f5}
	For $x>1/2$, define
	$$f_4(x)=\frac{e^{\pi ix}(7x-2)\Gamma(x+1)^3}{(2x-1)x^2\Gamma(3x+1)}.$$
	Then
	\begin{equation}\sum_{k=1}^\infty f_4'(k)=\frac 72\zeta(3)-\frac{\pi^3}{12}i,\ \ \sum_{k=1}^\infty f_4''(k)=7\pi\zeta(3)i-15\zeta(4),
	\end{equation}
	\begin{equation}\sum_{k=1}^\infty f_4'''(k)=279\zeta(5)-21\pi^2\zeta(3)-\f23\pi^5i,
	\end{equation}
	\begin{equation}\sum_{k=1}^\infty f_4^{(4)}(k)=6\l(56\zeta(3)^2-75\zeta(6)\r)+4\pi i\l(29\zeta(5)-14\pi^2\zeta(3)\r),
	\end{equation}
	\begin{equation}\begin{aligned}\sum_{k=1}^\infty f_4^{(5)}(k)=&\ 360\l(127\zeta(7)+21\zeta(3)\zeta(4)-93\zeta(2)\zeta(5)\r)
			\\&\ +60\pi i\l(28\zeta(3)^2-111\zeta(6)\r),\end{aligned}\end{equation}
	\begin{equation}\begin{aligned}\f1{72}\sum_{k=1}^\infty f_4^{(6)}(k)=&\ 61\zeta(8)-84\zeta(2)\zeta(3)^2+1680\zeta(3)\zeta(5)-408\zeta(5,3)
			\\&\ +30\pi i\l(127\zeta(7)-62\zeta(2)\zeta(5)\r),\end{aligned}\end{equation}
	and
	\begin{equation}\begin{aligned}&\ \f1{168}{\sum_{k=1}^\infty f_4^{(7)}(k)}
			\\=&\ 5\l(17009\zeta(9)+224\zeta(3)^3-13716\zeta(2)\zeta(7)-336\zeta(3)\zeta(6)+5022\zeta(4)\zeta(5)\r)
			\\&\ +3\pi i\l(1680\zeta(3)\zeta(5)-408\zeta(5,3)-560\zeta(2)\zeta(3)^2-2039\zeta(8)\r).\end{aligned}\end{equation}
\end{thm}
\begin{remark}
	Note that
	$$\sum_{k=1}^\infty f_4(k)=\sum_{k=1}^\infty\frac{(-1)^k(7k-2)}{(2k-1)k^2\binom{2k}k\binom{3k}k}
	=-\frac{\pi^2}{12}=-\frac{\zeta(2)}2$$
	by \cite[Example 24]{CZ}.
\end{remark}

\begin{thm}\label{f6} For $x>1/2$, define
	$$f_5(x):=\frac{(30x-11)\Gamma(x)^4}{4x(2x-1)\Gamma(2x)^2}.$$
	Then
	\begin{align}\sum_{k=1}^\infty f_5'(k)&=-24\zeta(4),
		\\\sum_{k=1}^\infty f_5''(k)&=8(23\zeta(5)-2\zeta(2)\zeta(3)),
		\\\sum_{k=1}^\infty f_5^{(3)}(k)&=24\left(2\zeta(3)^2-49\zeta(6)\right),
		\\\sum_{k=1}^\infty f_5^{(4)}(k)&=96(177\zeta(7)-46\zeta(2)\zeta(5)-2\zeta(3)\zeta(4)),
		\\\sum_{k=1}^\infty f_5^{(5)}(k)&=640\left(9\zeta(3)\zeta(5)-3\zeta(2)\zeta(3)^2-12\zeta(5,3)-146\zeta(8)\right),
	\end{align}
	and
	\begin{equation}\begin{aligned}\frac{\sum_{k=1}^\infty f_5^{(6)}(k)}{960}&=3001\zeta(9)+8\zeta(3)^3-6\zeta(3)\zeta(6)
			\\&\ \ -138\zeta(4)\zeta(5)-1062\zeta(2)\zeta(7).
	\end{aligned}\end{equation}
\end{thm}

\begin{remark}
Note that
	$$\sum_{k=1}^\infty f_5(k)=\sum_{k=1}^\infty\frac{30k-11}{k^3(2k-1)\binom{2k}k^2}=4\zeta(3)$$
	by \cite[Example 11]{CZ}.
\end{remark}

\begin{thm}\label{f7} For $x>0$, define
	$$f_6(x)=\frac{e^{\pi ix}(56x^2-32x+5)\Gamma(3x+1)}{(2x-1)^2\Gamma(x)^3}.$$
	Then
	\begin{align}\sum_{k=1}^\infty f_6'(k)&=\f{\pi^4}3-4\pi i\zeta(3),
		\\\sum_{k=1}^\infty f_6''(k)&=8\pi^2\zeta(3)-288\zeta(5)+\f23\pi^5 i,
		\\\sum_{k=1}^\infty f_6'''(k)&=1410\zeta(6)-96\zeta(3)^2+16\pi i(\pi^2\zeta(3)-54\zeta(5)),
		\\\f{\sum_{k=1}^\infty f_6^{(4)}(k)}{576}&=36\zeta(2)\zeta(5)-3\zeta(3)\zeta(4)-73\zeta(7)
		+\f{\pi i}6\l(85\zeta(6)-4\zeta(3)^2\r),
	\end{align}
	\begin{equation}
		\begin{aligned}\f{\sum_{k=1}^\infty f_6^{(5)}(k)}{48}
			=&\ 1715\zeta(8)+240\zeta(2)\zeta(3)^2-480\zeta(3)\zeta(5)+48\zeta(5,3)
			\\&\ +60\pi i(24\zeta(2)\zeta(5)-73\zeta(7)),
		\end{aligned}
	\end{equation}
	and
	\begin{equation}
		\begin{aligned}&\ \f{\sum_{k=1}^\infty f_6^{(6)}(k)}{3840}
			\\=&\ 1971\zeta(2)\zeta(7)+12\zeta(3)\zeta(6)-486\zeta(4)\zeta(5)-8\zeta(3)^3-2959\zeta(9)
			\\&\ +\f{\pi i}4(1417\zeta(8)+48\zeta(2)\zeta(3)^2-144\zeta(3)\zeta(5)+144\zeta(5,3)).
		\end{aligned}
	\end{equation}
\end{thm}
\begin{remark}
   Note that
	$$\sum_{k=1}^\infty f_6(k)=\sum_{k=1}^\infty\f{(-1)^k(56k^2-32k+5)}{(2k-1)^2k^3\bi{2k}k\bi{3k}k}=-4\zeta(3)$$
	by \cite[Example 21]{CZ}.
\end{remark}

\begin{thm}\label{f10} For $k>0$, define
	$$f_{7}(k):=\frac{(21k-8)\Gamma(k)^6}{8\Gamma(2k)^3}.$$
	Then
	\begin{align}\sum_{k=1}^\infty f_{7}''(k)&=\frac{57}2\zeta(4),
		\\\sum_{k=1}^\infty f_{7}^{(3)}(k)&=18\pi^2\zeta(3)-324\zeta(5),
		\\\sum_{k=1}^\infty f_{7}^{(4)}(k)&=\frac{1959}2\zeta(6)-432\zeta(3)^2,
		\\\sum_{k=1}^\infty f_{7}^{(5)}(k)&=19440\zeta(2)\zeta(5)-540\zeta(3)\zeta(4)-31860\zeta(7),
	\end{align}
	and
	\begin{equation}\sum_{k=1}^\infty f_{7}^{(6)}(k)=31104\zeta(5,3)+38880\zeta(2)\zeta(3)^2-77760\zeta(3)\zeta(5)-\frac{84807}4\zeta(8),
	\end{equation}
	Also,
	\begin{equation}\begin{aligned}\frac{\sum_{k=1}^\infty f_{7}^{(7)}(k)}{630}=&\ 6372\zeta(2)\zeta(7)-213\zeta(3)\zeta(6)-324\zeta(4)\zeta(5)
			\\&\ -288\zeta(3)^3-8812\zeta(9).\end{aligned}\end{equation}
\end{thm}
\begin{remark}
   Note that
	$$\sum_{k=1}^\infty f_{7}(k)=\sum_{k=1}^\infty\frac{21k-8}{k^3\binom{2k}k^3}=\zeta(2)$$
	as obtained by D. Zeilberger \cite{Zeil} by the WZ method. Also, we have the identity
	$\sum_{k=1}^\infty f_{7}'(k)=-6\zeta(3)$
	which is equivalent to the formula
	$$\sum_{k=1}^\infty\f{(21k-8)(H_{2k-1}-H_{k-1})-7/2}{k^3\bi{2k}k^3}=\zeta(3)$$
	obtained by K. C. Au \cite{Au25a}.
\end{remark}

\begin{thm}\label{f11} For $x>0$, define
	$$f_{8}(x):=\frac{2(145x^2-104x+18)\Gamma(2x)\Gamma(x)^4}{9\Gamma(3x)^2}.$$
	Then
	\begin{align}\sum_{k=1}^\infty f_{8}'(k)&=-16\zeta(3),
		\\\sum_{k=1}^\infty f_{8}''(k)&=107\zeta(4),
		\\\sum_{k=1}^\infty f_{8}^{(3)}(k)&=20\left(4\pi^2\zeta(3)-81\zeta(5)\right),
		\\\sum_{k=1}^\infty f_{8}^{(4)}(k)&=21\left(415\zeta(6)-128\zeta(3)^2\right),
	\end{align}
	\begin{equation}\frac1{30}\sum_{k=1}^\infty f_{8}^{(5)}(k)=5400\zeta(2)\zeta(5)+272\zeta(3)\zeta(4)-10761\zeta(7),
	\end{equation}
	and
	\begin{equation}\begin{aligned}\frac2{15}\sum_{k=1}^\infty f_{8}^{(6)}(k)=&\ 90665\zeta(8)-138240\zeta(3)\zeta(5)
			\\&\ +44928\zeta(5,3)+53760\zeta(2)\zeta(3)^2,
	\end{aligned}\end{equation}
	Also,
	\begin{equation}\begin{aligned}\frac1{420}\sum_{k=1}^\infty f_{8}^{(7)}(k)=&\ 161415\zeta(2)\zeta(7)+2340\zeta(3)\zeta(6)+13770\zeta(4)\zeta(5)
			\\&\ -6272\zeta(3)^3-284755\zeta(9).
	\end{aligned}\end{equation}
\end{thm}
\begin{remark}
 Note that
	$$\sum_{k=1}^\infty f_{8}(k)=\sum_{k=1}^\infty\frac{145k^2-104k+18}{(2k-1)k^3\binom{2k}k\binom{3k}k^2}=\frac{\pi^2}3$$
	as conjectured by the second author (cf. \cite[Remark 4.13]{S23} and \cite{Q456443}) in 2023, and later confirmed by K. C. Au \cite{Au25a}.
\end{remark}

\begin{thm}\label{f12} For $x>1/2$, we define
	$$f_{9}(x):=\frac{e^{\pi ix}(410x^2-197x+24)\Gamma(2x)^3\Gamma(x)^2}{2(2x-1)\Gamma(4x)^2}.$$
	Then
	\begin{align}\sum_{k=1}^\infty f_{9}'(k)&=22\zeta(3)-i\frac{\pi^3}3,
		\\ \sum_{k=1}^\infty f_{9}''(k)&=44\pi\zeta(3)i-174\zeta(4),
		\\\sum_{k=1}^\infty f_{9}^{(3)}(k)&=4140\zeta(5)-1584\zeta(2)\zeta(3)-\frac{97}{15}\pi^5i,
	\end{align}
	\begin{equation}\sum_{k=1}^\infty f_{9}^{(4)}(k)=24\left(374\zeta(3)^2-763\zeta(6)\right)+80\pi i\left(207\zeta(5)-11\pi^2\zeta(3)\right),
	\end{equation}
	\begin{equation}\begin{aligned}\sum_{k=1}^\infty f_{9}^{(5)}(k)=&\ 720\left(2083\zeta(7)+209\zeta(3)\zeta(4)-1380\zeta(2)\zeta(5)\right)
			\\&\ +60\pi i\l(748\zeta(3)^2-2219\zeta(6)\r),
	\end{aligned}\end{equation}
	\begin{equation}\begin{aligned}\f1{108}\sum_{k=1}^\infty f_{9}^{(6)}(k)=&\ 1875\zeta(8)+57200\zeta(3)\zeta(5)-19840\zeta(5,3)-2992\zeta(2)\zeta(3)^2
			\\&\ +40\pi i\l(2083\zeta(7)-1150\zeta(2)\zeta(5)+33\zeta(3)\zeta(4)\r),
	\end{aligned}\end{equation}
	and
	\begin{equation}\begin{aligned}&\ \f1{3360}\sum_{k=1}^\infty f_{9}^{(7)}(k)
			\\=&\ 295695\zeta(9)-224964\zeta(2)\zeta(7)+58995\zeta(4)\zeta(5)
			\\&\ -1023\zeta(3)\zeta(6)+6358\zeta(3)^3
			\\&\ +\f{\pi i}8\l(102960\zeta(3)\zeta(5)-35712\zeta(5,3)-4488\zeta(2)\zeta(3)^2-43801\zeta(8)\r).
	\end{aligned}\end{equation}
\end{thm}
\begin{remark}
 Note that
	$$\sum_{k=1}^\infty f_{9}(k)=\sum_{k=1}^\infty\frac{(-1)^k(410k^2-197k+24)}{k^3(2k-1)\binom{2k}k\binom{4k}{2k}^2}
	=-\frac{\pi^2}3$$
	by \cite[Example 63]{CZ}.
\end{remark}

\begin{thm}\label{f13}  For $x>1$, define
	$$f_{10}(x)=\frac{(364x^2-227x+36)\Gamma(x)^6}{9(2x-1)\Gamma(3x)^2}.$$
	Then
	\begin{align}\sum_{k=1}^\infty f_{10}'(k)&=-36\zeta(4),
		\\ \sum_{k=1}^\infty f_{10}''(k)
		&=24(17\zeta(5)-2\zeta(2)\zeta(3)),
		\\\sum_{k=1}^\infty f_{10}^{(3)}(k)&=12\left(16\zeta(3)^2-261\zeta(6)\right),
		\\\sum_{k=1}^\infty f_{10}^{(4)}(k)&=288\left(265\zeta(7)+2\zeta(3)\zeta(4)-102\zeta(2)\zeta(5)\right),
		\\\frac1{2880}\sum_{k=1}^\infty f_{10}^{(5)}(k)&=16\zeta(3)\zeta(5)-99\zeta(8)-24\zeta(5,3)-8\zeta(2)\zeta(3)^2,
	\end{align}
	and
	\begin{equation}\begin{aligned}\frac1{960}\sum_{k=1}^\infty f_{10}^{(6)}(k)=&\ 25849\zeta(9)+128\zeta(3)^3+30\zeta(3)\zeta(6)
			\\&\ +918\zeta(4)\zeta(5)-14310\zeta(2)\zeta(7).
	\end{aligned}\end{equation}
\end{thm}
\begin{remark}
   Note that
	$$\sum_{k=1}^\infty f_{10}(k)=\sum_{k=1}^\infty\frac{364k^2-227k+36}{k^4(2k-1)\binom{2k}k^2\binom{3k}k^2}=4\zeta(3)$$
	by \cite[Example 118]{CZ}.
\end{remark}

\begin{thm}\label{f14}  For $x>1/2$, define
	$$f_{11}(x)=4e^{\pi ix}(112x^3-8x^2-6x+1)\f{\Gamma(3x)\Gamma(2x-1)^3}{x^2\Gamma(x)^3\Gamma(6x)}.$$
	Then
	\begin{align}\sum_{k=1}^\infty f_{11}'(k)&=44\zeta(3)-\f23\pi^3i,
		\\\sum_{k=1}^\infty f_{11}''(k)&=44(2\zeta(3)\pi i-9\zeta(4)),
		\\\sum_{k=1}^\infty f_{11}'''(k)&=9696\zeta(5)-528\pi^2\zeta(3)-\f{218}{15}\pi^5i,
		\\\sum_{k=1}^\infty f_{11}^{(4)}(k)&=672(44\zeta(3)^2-119\zeta(6)+32\pi i(1212\zeta(5)-55\pi^2\zeta(3)),
	\end{align}
	\begin{equation}
		\begin{aligned}\sum_{k=1}^\infty f_{11}^{(5)}(k)=&\ 5760\l(905\zeta(7)-404\zeta(2)\zeta(5)-110\zeta(3)\zeta(4)\r)
			\\&\ +840\pi i\l(176\zeta(3)^2-587\zeta(6)\r),
		\end{aligned}
	\end{equation}
	\begin{equation}
		\begin{aligned}\f{\sum_{k=1}^\infty f_{11}^{(6)}(k)}{2304}=&\ 20460\zeta(3)\zeta(5)-1564\zeta(5,3)-4620\zeta(2)\zeta(3)^2-17197\zeta(8)
			\\&\ +5\pi i\l(2715\zeta(7)-1010\zeta(2)\zeta(5)-462\zeta(3)\zeta(4)\r),
		\end{aligned}
	\end{equation}
	and
	\begin{equation}
		\begin{aligned}&\ \f1{26880}\sum_{k=1}^\infty f_{11}^{(7)}(k)
			\\=&\ 225097\zeta(9)+4312\zeta(3)^3
			-9774\zeta(2)\zeta(7)
			\\&\ -21483\zeta(3)\zeta(6)-36360\zeta(4)\zeta(5)
			\\&\ +\f{\pi i}5\l(61380\zeta(3)\zeta(5)-4692\zeta(5,3)-11550\zeta(2)\zeta(3)^2-64921\zeta(8)\r).
		\end{aligned}
	\end{equation}
\end{thm}
\begin{remark}
 Note that
	$$\sum_{k=1}^\infty f_{11}(k)=\sum_{k=1}^\infty \f{(-1)^k(112k^3-8k^2-6k+1)\bi{2k}k^2}{k^2(2k-1)^3\bi{3k}k\bi{6k}{3k}}=-\f23\pi^2$$
	by \cite[Example 93]{CZ}.
\end{remark}

\begin{thm}\label{f15} For $x>0$, define
	$$f_{12}(x):=\f{e^{\pi ix}(205x^2-160x+32)\Gamma(x)^{10}}{32\Gamma(2x)^5}.$$
	Then
	\begin{align}\sum_{k=1}^\infty f_{12}'(k)&=15\zeta(4)-2\zeta(3)\pi i,
		\\\sum_{k=1}^\infty f_{12}''(k)&=\f{16}3\pi^2\zeta(3)-140\zeta(5)+\f{\pi^5}3i,
		\\\sum_{k=1}^\infty f_{12}'''(k)&=60(4\zeta(6)-\zeta(3)^2)+12\pi i(\pi^2\zeta(3)-35\zeta(5)),
	\end{align}
	\begin{equation}\begin{aligned}
			\sum_{k=1}^\infty f_{12}^{(4)}(k)=&\ 240(56\zeta(2)\zeta(5)-11\zeta(3)\zeta(4)-69\zeta(7))
			\\&\ +\f{4\pi i}{63}\l(37\pi^6-3780\zeta(3)^2\r),
	\end{aligned}\end{equation}
	\begin{equation}\begin{aligned}
			\sum_{k=1}^\infty f_{12}^{(5)}(k)=&\ 10(1536\zeta(5,3)-720\zeta(3)\zeta(5)+160\pi^2\zeta(3)^2-6157\zeta(8))
			\\&\ +24\pi i\l(350\pi^2\zeta(5)-3\pi^4\zeta(3)-3450\zeta(7)\r),
	\end{aligned}\end{equation}
	and
	\begin{equation}\begin{aligned}
			&\ \sum_{k=1}^\infty f_{12}^{(6)}(k)\\=&\ 480(8280\zeta(2)\zeta(7)+288\zeta(3)\zeta(6)-5775\zeta(4)\zeta(5)-50\zeta(3)^3-6565\zeta(9))
			\\&\ +180\pi i\l(40\pi^2\zeta(3)^2-240\zeta(3)\zeta(5)+512\zeta(5,3)-679\zeta(8)\r).
	\end{aligned}\end{equation}
\end{thm}
\begin{remark}
 Note that
	$$\sum_{k=1}^\infty f_{12}(k)=\sum_{k=1}^\infty \f{(-1)^k(205k^2-160k+32)}{k^5\bi{2k}{k}^5}=-2\zeta(3)$$
	by T. Amdeberhan and D. Zeilberger \cite{AZ97}.
\end{remark}

\begin{thm}\label{f16} For $x>1/4$, define
	$$f_{13}(x):=\frac{(60x^2-43x+8)\Gamma(4x)\Gamma(x)^8}{16(4x-1)\Gamma(2x)^6}.$$
	Then
	\begin{align}\sum_{k=1}^\infty f_{13}'(k)&=-8\zeta(3),
		\\ \sum_{k=1}^\infty f_{13}''(k)&=24\zeta(4),
		\\ \sum_{k=1}^\infty f_{13}^{(3)}(k)&=-48\zeta(5),
		\\\frac1{48}\sum_{k=1}^\infty f_{13}^{(4)}(k)&=16\zeta(3)^2-25\zeta(6),
		\\ \frac1{2880}\sum_{k=1}^\infty f_{13}^{(5)}(k)&=19\zeta(7)-14\zeta(3)\zeta(4),
		\\\frac1{5760}\sum_{k=1}^\infty f_{13}^{(6)}(k)&=32\zeta(5,3)+168\zeta(3)\zeta(5)-215\zeta(8),
	\end{align}
	and
	\begin{equation} \frac1{13440}\sum_{k=1}^\infty f_{13}^{(7)}(k)
		= 2489\zeta(9)-1860\zeta(3)\zeta(6)-32\zeta(3)^3-126\zeta(4)\zeta(5).
	\end{equation}
\end{thm}
\begin{remark}
 Note that $$\sum_{k=1}^\infty f_{13}(k)=\sum_{k=1}^\infty\frac{(60k^2-43k+8)\binom{4k}{2k}}{(4k-1)k^3\binom{2k}k^4}=\frac{\pi^2}3$$
	by \cite[Example 14]{CZ}.
\end{remark}

\begin{thm}\label{f17} For $x>1/6$, define
	$$ f_{14}(x):=\frac{(69x^2-40x+6)\Gamma(6x)\Gamma(x)^7}{18(6x-1)\Gamma(3x)^3\Gamma(2x)^2}.$$
	Then
	\begin{align}\sum_{k=1}^\infty f_{14}'(k)&=-12\zeta(3),
		\\ \sum_{k=1}^\infty f_{14}''(k)&=28\zeta(4),
		\\\sum_{k=1}^\infty f_{14}^{(3)}(k)&=72(7\zeta(5)-4\zeta(2)\zeta(3)),
	\end{align}
	\begin{align}
		\sum_{k=1}^\infty f_{14}^{(4)}(k)&=48\l(126\zeta(3)^2-181\zeta(6)\r),
		\\\sum_{k=1}^\infty f_{14}^{(5)}(k)&=720\l(623\zeta(7)+56\zeta(2)\zeta(5)-554\zeta(3)\zeta(4)\r),
	\end{align}
	\begin{equation}
		\begin{aligned}\f{1}{24}\sum_{k=1}^\infty f_{14}^{(6)}(k)=&\ 121824\zeta(5,3)+503280\zeta(3)\zeta(5)
			\\&\ +30240\zeta(2)\zeta(3)^2-701863\zeta(8),
	\end{aligned}\end{equation}
	and
	\begin{equation}
		\begin{aligned}\f{1}{20160}\sum_{k=1}^\infty f_{14}^{(7)}(k)=&\ 22657\zeta(9)+3738\zeta(2)\zeta(7)
			+5817\zeta(4)\zeta(5)
			\\&\ -27711\zeta(3)\zeta(6)-882\zeta(3)^3.
	\end{aligned}\end{equation}
\end{thm}
\begin{remark}
   Note that
	$$\sum_{k=1}^\infty f_{14}(k)= \sum_{k=1}^\infty\frac{(69k^2-40k+6)\binom{6k}{3k}}{(6k-1)k^3\binom{2k}k^3\binom{3k}k}=\frac23\pi^2$$
	by \cite[Example 32]{CZ}.
\end{remark}

\begin{thm}\label{f18} For $x>1/6$, define
	$$f_{15}(x)=\f{(74x^3-47x+8)\Gamma(6x-1)\Gamma(x)^9}{32\Gamma(3x)\Gamma(2x)^6}.$$
	Then
	\begin{align}\sum_{k=1}^\infty f_{15}'(k)&=-18\zeta(3),
		\\\sum_{k=1}^\infty f_{15}''(k)&=18\zeta(4),
		\\\sum_{k=1}^\infty f_{15}'''(k)&=36(35\zeta(5)-3\pi^2\zeta(3)),
		\\\sum_{k=1}^\infty f_{15}^{(4)}(k)&=36\l(300\zeta(3)^2-437\zeta(6)\r),
		\\\f1{2160}\sum_{k=1}^\infty f_{15}^{(5)}(k)&=305\zeta(7)+70\zeta(2)\zeta(5)-327\zeta(3)\zeta(4)),
		\\\f1{540}\sum_{k=1}^\infty f_{15}^{(6)}(k)&=600\pi^2\zeta(3)^2+9856\zeta(5,3)+35280\zeta(3)\zeta(5)-52735\zeta(8),
	\end{align}
	and
	\begin{equation}\begin{aligned}\f1{5040}\sum_{k=1}^\infty f_{15}^{(7)}(k)&=116890\zeta(9)-7500\zeta(3)^3-192789\zeta(3)\zeta(6)
			\\&\ +68670\zeta(4)\zeta(5)+32940\zeta(2)\zeta(7).
	\end{aligned}\end{equation}
\end{thm}
\begin{remark}
   Note that
	$$\sum_{k=1}^\infty f_{15}(k)=\sum_{k=1}^\infty \f{(74k^2-47k+8)\bi{6k}{3k}\bi{3k}k}{(6k-1)k^3\bi{2k}k^5}=2\pi^2$$
	by \cite[Example 52]{CZ}.
\end{remark}

\section{Evaluating hypergeometric series}
In this section, we present our proofs of Theorems \ref{th1}-\ref{th2}.
Since for a WZ pair $(F(n,k), G(n,k))$, $G(n,k)$ can be computed by Gosper's algorithm from $F(n,k)$, we only provide the explicit expression of $F(n,k)$ in what follows.

\noindent {\it Proof of Theorem~\ref{th1}.}
We find a WZ pair with
\[
F(n,k) = \frac { \left( -1 \right) ^{n}{16}^{-n} \left( \frac{1}{2}+n+k \right) !\,k!
	\,    \left( 2\,n \right) !  ^{4}}{    \left( 2\,n+1+
	k \right) !  ^{2} \left( -\frac{1}{2}+3\,n \right) !\,  n!
	 ^{2}}.
\]
In this case, \eqref{keyid} becomes
\[
\sum_{n=0}^\infty G(n,0) = \sum_{k=0}^\infty F(0,k).
\]
It can be checked that $G(n,0)$ is exactly the summand of the series on the left hand side of \eqref{id-1} by making the substitution $k \to n+1$. Note that
\[
\sum_{k=0}^\infty F(0,k) = \sum_{k=0}^\infty \frac {{2}^{-1-2\,k} \left( 2\,k+1 \right)  \left( 2\,k \right) !}{
	\left( k+1 \right) ^{2}  k!  ^{2}}=2\log2\]
 by {\tt WolframAlpha}. \qed

\noindent {\it Proof of Theorem~\ref{th2}.}
We use the following WZ pair with parameters $a$ and $c$:
\begin{multline}
F(n,k) = 6 \left( -1 \right) ^{-a-3\,k-3\,n+c}{2}^{-2\,c+2\,k}
		 \Gamma  \left( -a+n+{\textstyle \frac{1}{2}} \right)  ^{2}\Gamma  \left( c-
		{\textstyle \frac{1}{2}}-k \right) \\[5pt]
	\times \frac
 { \Gamma  \left( \frac{1}{2}+k+n \right) ^{2}
		\Gamma  \left( a+\frac{1}{2}+n+2\,c \right) \Gamma  \left( \frac{1}{2}-a+n-c+k
		\right)}{\pi^3\, \Gamma  \left( -a+n \right)  \Gamma
		\left( c+n \right) ^{2}\Gamma  \left( -a+3\,n+\frac{3}{2}+2\,k
		\right). }
\end{multline}
In this case, $F(0,k) \to 0$ as $a,c \to 0$, and
\begin{multline*}
g(n) := \lim_{k \to \infty} G(n,k) = 12 \sqrt{2} \left( -1 \right) ^{a+1+n} {2}^{-4+3\,a-6\,c-9\,n}  \left( 2\,a-4\,c-6\,n-1 \right)   \\
\times \frac{ \left( -2\,a+2\,n \right) !^{2} \left( 2\,a+2\,n+4\,c \right) !}{\left( -a+n
		\right) !^{3} \left( a+n+2\,c \right) !\, \left( c+n \right) ! ^{2}}.
\end{multline*}
Hence \eqref{keyid} becomes
\[
\sum_{n=0}^\infty G(n,0) = \sum_{n=0}^\infty g(n).
\]
Taking the coefficient of $a^0 c^0$ (denoted by $[a^0 c^0]$), we derive
\[
 \sum_{n=0}^{\infty} \frac{(28n^2 + 10n + 1) \binom{2n}{n}^5}{(6n + 1)(-64)^n \binom{3n}{n} \binom{6n}{3n}}
=  \frac{3}{4} \, \sum_{n=0}^\infty {\frac { \left( -1 \right) ^{n}\sqrt {2} \, {512}^{-n} \left( 6\,n+1
		\right)   \left( 2\,n \right) ! ^{3}}{  n! ^{6}}} \\
= \frac{3}{\pi},
\]
where the second equality is a well-known identity of Ramanujan (cf. \cite[Identity 3]{Gui08}).

Taking the real part of the coefficient $[c]$, we get
\begin{align*}
& \sum_{n=0}^{\infty} \frac{\big((28n^2 + 10n + 1)(2H_{2n} - 3H_n - 6 \log 2) + 20n + 4\big) \binom{2n}{n}^5}{(6n + 1)(-64)^n \binom{3n}{n} \binom{6n}{3n}}  \\
 = &\
3 \sqrt{2} \sum_{n=0}^\infty  \frac{(2n)!^3}{n!^6} \l( (6n+1)\l(H_{2n}-H_n - \frac{3}{2} \log 2\r) +1 \r).
\end{align*}
 Note that
\[
\sum_{n=0}^\infty  \frac{(2n)!^3}{n!^6} \big( (6n+1)(H_{2n}-H_n) +1 \big) = \frac{3 \sqrt{2} \log 2}{\pi}
\]
by taking derivative on both sides of \cite[Indentity 3]{Gui08}.
(For a general Ramanujan-type series, its variant involving harmonic numbers
was conjectured by Sun \cite{harmonic} and confirmed by Y. Zhou \cite{ZhouRJ}.) Hence
\begin{align*}
& \sum_{n=0}^{\infty} \frac{\big((28n^2 + 10n + 1)(2H_{2n} - 3H_n ) + 20n + 4\big) \binom{2n}{n}^5}{(6n + 1)(-64)^n \binom{3n}{n} \binom{6n}{3n}} 
= 6 (\log 2) \frac{3}{\pi} = \frac{18 \log 2}{\pi}.
\end{align*}

Taking the real part of the linear combination $2[c]+[a]$ of the coefficients, we get
\begin{align*}
& \sum_{n=0}^{\infty} \frac{\big((28n^2 + 10n + 1)(2H_{6n} - H_{3n} - 3H_n - 10 \log 2) + f(n)\big) \binom{2n}{n}^5}{(6n + 1)(-64)^n \binom{3n}{n} \binom{6n}{3n}} \\
=\, & \frac{9 \sqrt{2}}{2}  3 \sqrt{2} \sum_{n=0}^\infty  \frac{(2n)!^3}{n!^6} \l( (6n+1)\l(H_{2n}-H_n - \frac{3}{2} \log 2\r) +1 \r) = 0.
\end{align*}
Therefore,
\begin{align*}
& \sum_{n=0}^{\infty} \frac{\big((28n^2 + 10n + 1)(2H_{6n} - H_{3n} - 3H_n - 10 \log 2) + f(n)\big) \binom{2n}{n}^5}{(6n + 1)(-64)^n \binom{3n}{n} \binom{6n}{3n}} \\
 &\qquad=\, 10 (\log 2) \frac{3}{\pi} = \frac{30 \log 2}{\pi}.
\end{align*}
This ends our proof. \qed

\section{Proofs of Theorems \ref{f1}-\ref{f18}}

In this section, we provide our proofs of Theorems \ref{f1}--\ref{f18}. We will give a detailed proof for Theorem~\ref{f3} and only provide the WZ pairs for the remaining theorems.

\noindent {\it Proof of Theorem~\ref{f3}.}
We find a WZ pair with
\[
F(n,k) = -\frac{2}{5} \,{\frac { \left( -1 \right) ^{n}  n! ^{2}k!}{
		\left( k+n+1 \right) ^{2} \left( 2\,n+1+k \right)  \left( 2\,n+k
		\right) !}}.
\]
We have
\[
\sum_{n=0} f_2^{(m)}(n) = \sum_{n=0}^\infty \frac{\partial^m }{n^m}  G(n,0)
= \sum_{k=0}^\infty \frac{\partial^m }{n^m} F(0, k).
\]
When $m=2$, we get
\begin{multline}\label{Hsum}
\frac{\partial^2 }{n^2} F(0,k) = - \frac{8}{5} \left( \frac{H_k^{(2)} + H_k^2 - \pi^2/3}{(k+1)^3} + \frac{4 H_k}{(k+1)^4}  \right.\\
\left. + \frac{11}{2(k+1)^5}  - \frac{i \pi H_k}{(k+1)^3} - \frac{2i \pi}{(k+1)^4} \right).
\end{multline}
By the {\tt Mathematica} package {\tt MultipleZetaValue} developed by Au \cite{Au25b}, the sum on the right hand side equals
\[
\frac{8 \pi ^2 \zeta (3)}{15}-\frac{52 \zeta (5)}{5}+\frac{i \pi ^5}{25},
\]
which confirms \eqref{3-1}. In a similar way, we are able to evaluate $\sum_{n=0}^\infty f_2^{(m)}(n)$ for $m=3$ to $6$, which corresponds to \eqref{3-2}--\eqref{3-5}. \qed

\noindent {\it Proof of Theorem~\ref{f1}.}
Take
\[
	F(n,k) = \frac{1}{3} \,{\frac { \left( k+n \right) !\,n!}{ \left( k+1+n \right)  \left(
			1+2\,n+k \right)  \left( 2\,n+k \right) !}}.
\]

\noindent {\it Proof of Theorem~\ref{f19}.}
Take
\[
F(n,k)=8\,{\frac {{3}^{3\,k}n  \left( 3\,n \right) ! ^{2}
		\left( 2\,n+k \right) !\,   \left( k+n \right) ! ^{2}}{
		\left( 1+6\,n+3\,k \right)  \left( 6\,n+2+3\,k \right)  n!
		 ^{2} \left( 6\,n+3\,k \right) !\, \left( 2\,n \right) !}}.
\]

\noindent {\it Proof of Theorem~\ref{f5}.}
Take
{\small
\[
F(n,k) = -\frac{1}{2} \,{\frac { \left( n+1 \right)  \left( -1 \right) ^{n} \left( 3+4\,
		n+4\,k \right)  \left( n+2\,k \right) !\, n!^{2}}{
		\left( 2\,k+1+2\,n \right)  \left( 1+3\,n+2\,k \right)  \left( 2+3\,n
		+2\,k \right)  \left( k+n+1 \right)  \left( 3\,n+2\,k \right) !}}.
\]}

\noindent {\it Proof of Theorem~\ref{f6}.}
Take
\[
F(n,k) = \,{\frac {4  k!^{2}  n!^{4}}{ \left( k+n
		+1 \right)  \left( 2\,n+1+k \right) ^{2} \left(   2\,n+k
		\right) ! ^{2}}}.
\]

\noindent {\it Proof of Theorem~\ref{f7}.}
Take
\[
F(n,k) = \,{\frac {4 \left( -1 \right) ^{n+1}n! ^{2} \left( k+n
		\right) !}{ \left( 1+3\,n+k \right) 2\,n+1+k ^{2}
		\left( 3\,n+k \right) !}}.
\]

\noindent {\it Proof of Theorem~\ref{f10}.}
Take
\[
F(n,k) = {\frac {   \left( k+n \right) ! ^{2}  n! ^
		{4}}{  2\,n+1+k ^{2}   \left( 2\,n+k \right) !
	  ^{2} \left( 2\,n \right) !}}.
\]

\noindent {\it Proof of Theorem~\ref{f11}.}
Take
\begin{multline*}
F(n,k) = \,{\frac { 2 \left( 5\,n+1 \right)  \left( 3+6\,n+4\,k \right)  }{ \left( 2\,k+1+2\,n \right)  \left( 5\,n+1+2\,
		k \right)  \left( 5\,n+2+2\,k \right)  \left( 1+3\,n+k \right) }} \\[5pt]
\times
\frac {  \left(
	2\,n \right) !\, \left( 2\,n+k \right) !\, \left( n+2\,k \right) !\,
	\left( 2\,k \right) !\, \left( 5\,n \right) !\,   n!   ^{2
	} \left( k+n \right) !}{ k!\,
	\left( 2\,k+2\,n \right) !\, \left( 3\,n \right) !\, \left( 5\,n+2\,k
	\right) !\, \left( 3\,n+k \right) !}.
\end{multline*}

\noindent {\it Proof of Theorem~\ref{f12}.}
Take
\begin{multline*}
	F(n,k) =\left( -1 \right) ^{n+1}\frac { \left( 2+3\,n+2\,k \right)  \left( 2\,k+1+2\,n \right)
		 		}{ \left( 4\,n+1+2\,k \right)  \left( 2\,n+1+k \right) ^{3} } \\[5pt]
\times \frac {   \left( 2\,n \right) !  ^{4}
	 \left( k+n \right) ! ^{2} \left( 2\,k+2\,n \right) !
}{
	\left( 2\,n+k \right) !  ^{2} \left( 4\,n+2\,k \right) !\,
	\left( 4\,n \right) !}.
\end{multline*}

\noindent {\it Proof of Theorem~\ref{f13}.}
Take
\[
	F(n,k) = \,{\frac {2 \left( 2+3\,n+2\,k \right) k!\,   \left( k+n
			\right) !  ^{3}   \left( 2\,n \right) !  ^{3}
			  n!   ^{3}}{ \left( 1+3\,n+k \right)  \left( 2\,n+1+k
			\right) ^{3} \left( 3\,n+k \right) !\,    \left( 2\,n+k \right)
			!   ^{3} \left( 3\,n \right) !}}.
\]

\noindent {\it Proof of Theorem~\ref{f14}.}
Take
\begin{multline*}
	F(n,k) = \left( -1 \right) ^{n+1} \,{\frac {2 \left( 3+6\,n+4\,k \right)  \left( 2\,n+1 \right)
			}{ \left( 6
			\,n+1+2\,k \right)  \left( 2\,n+1+k \right) ^{2} \left( 2\,k+1+2\,n
			\right) }} \\[5pt]
-2\,{\frac { k!\, \left( 2\,k+2\,n \right) !\,  (
		\left( 2\,n \right) !   ^{3} \left( 3\,n+k \right) !}{  \left( k+n \right) !\,   n! ^{2} \left( 2\,n+k
		\right) !\, \left( 6\,n+2\,k \right) !}}.
\end{multline*}

\noindent {\it Proof of Theorem~\ref{f15}.}
Take
\[
F(n,k) = {\frac { \left( -1 \right) ^{n+1} \left( 2+3\,n+2\,k \right)
		\left( k+n \right) !   ^{4}  n!   ^{6}}{ \left( 2\,
		n+1+k \right) ^{4}   \left( 2\,n+k \right) !  ^{4}
		\left( 2\,n \right) !}}.
\]

\noindent {\it Proof of Theorem~\ref{f16}.}
Take
\[
F(n,k)= \frac { \left( 2+3\,n+2\,k \right)  \left( 4\,n+1+2\,k \right)
 \left( k+n \right) ! ^{4}  n!   ^{4}
	\left( 4\,n+2\,k \right) !}{ \left( 2\,k+1+2\,n \right)  \left( 2\,n+
	1+k \right) ^{3}    \left( 2\,n+k \right) ! ^{4} \left( 2
	\,k+2\,n \right) !\, \left( 2\,n \right) !}.
\]

\noindent {\it Proof of Theorem~\ref{f17}.}
Take
\[
F(n,k)= {\frac { 4 \left( 6\,n+1+2\,k \right)    n!  ^{4} \left(
		6\,n+2\,k \right) !\,  \left( k+n \right) !  ^{3}}{
		\left( 2\,k+1+2\,n \right)  \left( 1+3\,n+k \right) ^{2}
		\left( 3\,n+k \right) !  ^{3} \left( 2\,n \right) !\, \left( 2
		\,k+2\,n \right) !}}.
\]

\noindent {\it Proof of Theorem~\ref{f18}.}
Take
\[
F(n,k)= -\,{\frac { 3 \left( 6\,n+1+2\,k \right)  \left( 2+3\,n+2\,k \right) k!
		\,  n!   ^{6}   \left( k+n \right) !  ^{3}
		\left( 6\,n+2\,k \right) !}{ \left( 2\,n+1+k \right) ^{3} \left( 2\,k
		+1 \right)    \left( 2\,n+k \right) ! ^{3} \left( 2\,k
		\right) !\, \left( 2\,n \right) !  ^{3} \left( 3\,n+k
		\right) !}}.
\]

\end{document}